\documentclass[12pt]{amsart}
\usepackage{mtymath}
\usepackage{iftex}
\usepackage[T1]{fontenc}
\usepackage{fvextra}

\usepackage{xcolor}
\definecolor{darkblue}{rgb}{0.0,0.0,0.65}
\definecolor{darkred}{rgb}{0.68,0.05,0.0}
\definecolor{darkgreen}{rgb}{0.0,0.29,0.29}
\definecolor{darkpurple}{rgb}{0.47,0.09,0.29}
\usepackage{hyperref}
\hypersetup{
   colorlinks = true,
   citecolor  = darkblue,
   linkcolor  = darkred,
   filecolor  = darkblue,
   urlcolor   = darkblue,
}

\newenvironment{extendednumbering}
{%
  \begingroup
  \let\baseequationnumber\theequation
  \renewcommand{\theequation}{\baseequationnumber*}%
}
{%
  \endgroup
}

\title{Powers of the Vandermonde determinant are eventually non-SNP}
\author{Thien Le}
\author{Melanie Weber}
\date{July 23, 2026}
\begin{document}

\begin{abstract}
We prove a conjecture of Monical, Tokcan, and Yong that every fixed
positive power of the Vandermonde determinant is non-SNP in all
sufficiently many variables, where a polynomial is
non-SNP if there is a lattice point in its Newton polytope that does not
appear with nonzero coefficient.  This means our result proves that for
every even power $k\geq4$, there is always such a missing lattice
monomial in large enough dimensions.  The odd case follows from
alternation, and the quadratic case was previously known.  For every
even power $k\geq4$, we exhibit an explicit lattice point in the Newton
polytope of
$a_{\delta_k}^k$ whose coefficient vanishes.  The vanishing is obtained
from a Dyson constant-term identity, proved using the finite-variable
Jack scalar product and Macdonald's specialization formula.  The key
even-power construction and proof strategy arose from prompting with
OpenAI Codex (GPT Sol 5.6 Extra High), a large language model; the
complete transcript appears in the appendix.
The authors subsequently checked and organized the argument.  The
accompanying Lean formalization is available at\newline
\url{https://github.com/steven-le-thien/vandermonde-snp}.
\end{abstract}

\maketitle

\section{Introduction}
Let $\mathbb{N}$ be the set of nonnegative integers. For a polynomial
$f(x)$ and an exponent $\alpha \in \mathbb{N}^n$, we write
$[x^\alpha]f$ for the coefficient of $x^\alpha$ in the monomial expansion
of $f$; this coefficient may be zero.

\subsection{Newton polytopes and SNP}
Given a polynomial
$f(x)=\sum_{\alpha\in\mathbb{N}^n}c_\alpha x^\alpha$, its
\emph{Newton polytope} is the convex hull of the exponents having nonzero
coefficient:
\begin{equation}
\Newton(f)=\conv\bigl(\{\alpha : c_{\alpha}\neq 0\}\bigr)\subseteq\RR^n.
\end{equation}

\begin{definition}
The polynomial $f$ has a \emph{saturated Newton polytope (SNP)} if
\[
    c_\alpha\neq 0
    \qquad\text{for every }\alpha\in\Newton(f)\cap\mathbb{N}^n.
\]
It is \emph{non-SNP} if there exists
$\alpha\in\Newton(f)\cap\mathbb{N}^n$ with $c_\alpha=0$
\cite[Definition~1.1]{MonicalTokcanYong2019}.
\end{definition}

This condition is useful in complexity theory because it links the geometric
description of a polynomial to the hardness of its coefficient structure.
Adve, Robichaux, and Yong develop this viewpoint for Schubert polynomials:
a tableau criterion based on the Schubitope gives the first polynomial-time
algorithm for deciding coefficient nonvanishing
\cite{AdveRobichauxYong2019Complexity}.  Their full treatment also proves
that computing a Schubert coefficient exactly is \#P-complete, despite the
tractability of deciding whether it vanishes
\cite{AdveRobichauxYong2021Vanishing}.
Newton polytopes can often be described combinatorially, while exact support
problems---deciding which coefficients vanish in high-dimensional expansions---
can be computationally expensive for general families.  When a family is
SNP, geometric membership in $\Newton(f)$ becomes enough to recover support,
which can significantly simplify symbolic tasks such as sparse interpolation,
support reconstruction, and zero-testing for structured circuits; proving that
a family is non-SNP identifies explicit obstructions to this simplification and
helps separate regimes where geometric approximations of complexity are sharp
from those where hidden cancellations remain.  As an example from complexity
theory, many coefficient-extraction tasks are known to be hard via reductions
to the permanent problem \cite{Valiant1979Complexity}.

\subsection{Vandermonde determinant}
The \emph{Vandermonde determinant} in $n$ variables, denoted by
$a_{\delta_n}$ in \cite{MonicalTokcanYong2019}, is
\begin{equation}
    a_{\delta_n}(x):=\prod_{1\leq i<j\leq n}(x_i-x_j),
    \qquad \delta_n=(n-1,n-2,\ldots,0).
\end{equation}

Monical, Tokcan, and Yong proved that $a_{\delta_n}$ is SNP if and only if
$n\leq2$ and that $a_{\delta_n}^2$, the discriminant, is SNP if and only
if $n\leq4$ \cite[Propositions~2.22 and~2.23]{MonicalTokcanYong2019}.
They further conjectured that every fixed positive power of the
Vandermonde determinant is eventually non-SNP
\cite[Conjecture~2.25]{MonicalTokcanYong2019}.  If true, it would give a
uniform asymptotic description of support behavior for all fixed powers;
if false, an eventual counterexample would pinpoint a new geometric
obstruction in Newton-polytope saturation and clarify how SNP phenomena
break down in larger rank.

Coefficient expansions of powers of the Vandermonde determinant have been
studied from several perspectives
\cite{ScharfThibonWybourne1994,Ballantine2012,BelbachirBoussicaultLuque2008}.
Dyson-type constant-term orthogonality identities also form a substantial
related literature \cite{Kadell2000,KarolyiLascouxWarnaar2015}.  The proof
below uses the finite-variable Jack scalar product to establish the
particular constant-term identity needed here.

We prove the conjecture.
\begin{theorem}\label{thm:main}
For every integer $k\geq1$, there is an integer $N_k$ such that
$a_{\delta_n}^k$ is non-SNP for every $n\geq N_k$.
\end{theorem}

\section{Key tools}
\subsection{Majorization criterion}
The Newton polytope of the Vandermonde determinant is the permutahedron
\cite[Proposition~2.3]{postnikov2005permutohedraassociahedra}:
\[
    \Newton(a_{\delta_n})=\mathcal{P}(n-1,n-2,\ldots,0)\subset\mathbb{R}^n,
\]
where, for $\beta\in\mathbb{N}^n$,
\begin{equation}
    \mathcal{P}(\beta)=\conv\{\pi\beta:\pi\in\mathcal{S}_n\}
    \subset\mathbb{R}^n.
\end{equation}
Here $\mathcal{S}_n$ is the symmetric group.  Since
$\Newton(fg)=\Newton(f)+\Newton(g)$
\cite[equation~(1.6)]{sturmfels1998polynomialEA}, it follows that
\[
    \Newton(a_{\delta_n}^k)=k\mathcal{P}(n-1,\ldots,0).
\]

For $\alpha,\beta\in\mathbb{N}^n$, let $\alpha^\downarrow$ and
$\beta^\downarrow$ denote their decreasing rearrangements.  We say that
$\beta$ majorizes $\alpha$, and write $\alpha\preceq\beta$, if
\begin{equation}
    \sum_{i=1}^r\alpha_i^\downarrow
    \leq\sum_{i=1}^r\beta_i^\downarrow
    \quad(1\leq r<n),
    \qquad
    \sum_{i=1}^n\alpha_i^\downarrow
    =\sum_{i=1}^n\beta_i^\downarrow.
    \label{eqn:majorization}
\end{equation}
Rado's theorem \cite[Theorem~1]{rado1952inequality} gives
\begin{equation}
    \alpha\in k\mathcal{P}(n-1,\ldots,0)
    \quad\Longleftrightarrow\quad
    \alpha\preceq k(n-1,\ldots,0).
\end{equation}

\subsection{Specialization map}
Let $\Lambda_{\mathbb{Q}}$ be the ring of symmetric functions over
$\mathbb{Q}$.  The power-sum symmetric functions
$p_j=\sum_i x_i^j$ generate
$\Lambda_{\mathbb{Q}}=\mathbb{Q}[p_1,p_2,\ldots]$
\cite[Chapter~I, \S2, equation~(2.12)]{macdonald1995symmetric}.  For a
scalar $c$, define the specialization
\begin{equation}
    \epsilon_c:\Lambda_{\mathbb{Q}}\longrightarrow\mathbb{Q}(c),
    \qquad \epsilon_c(p_j)=c\quad(j\geq1).
\end{equation}
When $c=q$ is a nonnegative integer,
\[
    \epsilon_q(f)=f(\underbrace{1,\ldots,1}_{q\text{ times}},0,0,\ldots),
\]
but the algebraic specialization is also defined for negative and
nonintegral $c$.  We will use Macdonald's specialization formula for Jack
polynomials below.  We write $f(Z)$ for evaluation at
$Z=(z_1,\ldots,z_s,0,0,\ldots)$.

\subsection{Finite-variable Jack polynomials}
For a partition $\lambda=(\lambda_1\geq\cdots\geq\lambda_\ell>0)$,
write $\ell(\lambda)=\ell$ for its length and write $\lambda\vdash r$
when $\sum_i\lambda_i=r$.  Let $m_\lambda$ denote the monomial
symmetric function obtained by summing the distinct monomials whose
exponent vectors are permutations of $\lambda$ with trailing zeros.
Fix positive integers
$s,m$ and set $\tau=1/m$.  We write $P_\lambda^{(\tau)}$ for the Jack
symmetric function with parameter $\tau$ in Macdonald's monic
$P$-normalization: its coefficient of $m_\lambda$ is $1$, and its
remaining monomial terms are indexed by partitions dominated by
$\lambda$
\cite[Chapter~VI, \S10, equations~(10.13)--(10.14)]{macdonald1995symmetric}.

For a Laurent polynomial $G$ in $z_1,\ldots,z_s$, define
\[
    \CT(G):=[z_1^0\cdots z_s^0]G,
\]
and write $Z^{-1}=(z_1^{-1},\ldots,z_s^{-1})$.  We use only two standard
facts about Jack functions.  First, the evaluated polynomials
$P_\lambda^{(\tau)}(Z)$ with $\ell(\lambda)\leq s$ form an orthogonal
basis for the finite-variable scalar product on symmetric polynomials
$g,h$ in $Z$
\cite[Chapter~VI, \S10, equations~(10.34)--(10.36)]{macdonald1995symmetric}:
\begin{equation}
\label{eq:jack-scalar}
\langle g,h\rangle'_{s,\tau}
=\frac1{s!}\CT\left(
g(Z)h(Z^{-1})
\prod_{i\neq j}\left(1-\frac{z_i}{z_j}\right)^{1/\tau}
\right).
\end{equation}
The monic triangular normalization also gives
$P_{(1^s)}^{(\tau)}=m_{(1^s)}=e_s$, where $(1^s)$ is the partition
with $s$ parts equal to $1$ and $e_s$ is the $s$th elementary
symmetric function.  In particular,
$e_s(Z)=z_1\cdots z_s$.  Since $\tau=1/m$, the exponent $1/\tau$ in
\eqref{eq:jack-scalar} is the positive integer $m$.

Second, the specialization map from the preceding subsection satisfies
Macdonald's formula.  Let $\lambda'$ denote the conjugate partition, so
that $\lambda'_j$ is the number of cells in column $j$ of the Young
diagram of $\lambda$.  For a cell $u=(i,j)$ in that diagram, set
\[
    a(u)=\lambda_i-j,\qquad
    \ell(u)=\lambda'_j-i,\qquad
    a'(u)=j-1,\qquad
    \ell'(u)=i-1;
\]
these are respectively the arm, leg, coarm, and coleg of $u$.  Then
\begin{equation}
\label{eq:jack-specialization}
\epsilon_{-\tau}\left(P_\lambda^{(\tau)}\right)
=
\prod_{u\in\lambda}
\frac{-\tau+\tau a'(u)-\ell'(u)}
     {\tau a(u)+\ell(u)+1},
\end{equation}
where the product is over the cells $u$ of $\lambda$
\cite[Chapter~VI, \S10, equation~(10.20)]{macdonald1995symmetric}.

\subsection{A Dyson constant-term identity}

For integers $s,m\geq1$, define the Laurent polynomial
\begin{equation}
    D_{s,m}(Z) := \prod_{i \neq j} (1 - z_i/z_j)^m,
\end{equation}

Pairing the factors indexed by $(i,j)$ and $(j,i)$ gives
\begin{equation}\label{eq:even_power_vandermonde_as_dyson}
    a_{\delta_s}(Z)^{2m}
=
(-1)^{m\binom{s}{2}}
e_s(Z)^{m(s-1)}D_{s,m}(Z).
\end{equation}

\begin{lemma}[Dyson constant-term lemma]\label{lem:dyson}
Let $f\in\Lambda_{\mathbb{Q}}$ be homogeneous of degree $s$, and write
$(a)_s:=a(a+1)\cdots(a+s-1)$ for the rising factorial.  Then
    \begin{equation}
        \frac{\CT\!\left(D_{s,m}(Z)f(Z)/e_s(Z)\right)}
             {\CT(D_{s,m}(Z))}
        =
        \frac{(-1)^s s!}{(1/m)_s}\epsilon_{-1/m}(f),
        \label{eq:dyson-lemma}
    \end{equation}
\end{lemma}
We postpone the proof of the lemma to a later section.

\section{Known results}
If $a_{\delta_n}^k$ is non-SNP, then $a_{\delta_{n+1}}^k$ is also
non-SNP \cite[Lemma~2.24]{MonicalTokcanYong2019}.

For the first four odd powers, Monical, Tokcan, and Yong record the same
three-variable threshold and witness used below
\cite[Example~2.21]{MonicalTokcanYong2019}.  Through personal
communication, we believe that Yan X.\ Zhang was the first to observe
that the odd case is immediate from skew-symmetry.  We include a proof
below for completeness.
\begin{fact}
For every odd positive integer $k$, $a_{\delta_3}^k$ is non-SNP.  In
particular, the exponent
\[
    \alpha=\left(\frac{3k-1}{2},\frac{3k-1}{2},1\right)
\]
lies in its Newton polytope and has coefficient zero.
\end{fact}
\begin{proof}
The vector $\alpha$ is nonincreasing.  Its coordinates sum to
$3k=k(2+1+0)$, and
\[
    \alpha_1+\alpha_2=3k-1\leq3k,
    \qquad
    \alpha_1=\frac{3k-1}{2}\leq2k.
\]
The majorization criterion \eqref{eqn:majorization} therefore gives
$\alpha\in\Newton(a_{\delta_3}^k)$.

Because $k$ is odd, $a_{\delta_3}^k$ is skew-symmetric.  Swapping $x_1$
and $x_2$ negates the polynomial but fixes the monomial $x^\alpha$,
since $\alpha_1=\alpha_2$.  Hence $[x^\alpha]a_{\delta_3}^k=0$.
\end{proof}

We also recall the quadratic case, which is not covered by the even-power
construction below.
\begin{fact}
$a_{\delta_n}^2$ is SNP if and only if $n\leq4$
\cite[Proposition~2.23]{MonicalTokcanYong2019}.
\end{fact}

\section{Main result}
\begin{theorem}[Even powers]\label{thm:even-seed}
For every even integer $k\geq4$, $a_{\delta_k}^k$ is non-SNP.
\end{theorem}
\begin{proof}
We exhibit an integral point in the Newton polytope whose coefficient
vanishes.

\smallskip
\noindent\textbf{Part 1: the exponent.}
Write $k=2m$, where $m\geq2$, and set
    \begin{align*}
        &N=2m,\qquad s=N-2,\qquad q=N-1,\\
        &H=q^2,\qquad L=(m-1)q=m(s-1)+1.
    \end{align*}

    The exponent in question is
    \begin{equation}
        \alpha=(H,H,\underbrace{L,\ldots,L}_{s\text{ times}}).
    \end{equation}

\smallskip
\noindent\textbf{Part 2: membership in the Newton polytope.}
Recall that
\[
    \Newton(a_{\delta_k}^k)=k\mathcal{P}(k-1,\ldots,0).
\]

First observe that
\begin{equation}
    2H+sL
=
2q^2+(2m-2)(m-1)q
=
2m^2q
=
k\binom{k}{2}.
\end{equation}

The mean of the coordinates of both $\alpha$ and
$k\cdot(k-1,k-2,\ldots,0)$ is $mq$.  After subtracting this mean, the
decreasing rearrangement of $\alpha$ is
\begin{equation}
\bar\alpha=
\bigl((m-1)q,(m-1)q,
\underbrace{-q,\ldots,-q}_{N-2\text{ times}}\bigr).
\end{equation}
The majorization inequalities for membership in \(k \mathcal{P}(k - 1, \ldots, 0)\)
become
\begin{equation}
\sum_{i=1}^r {\bar \alpha}_i\leq mr(k-r),
\qquad 1\leq r<k.
\end{equation}
For $r=1$, this is
$(m-1)q\leq mq$.
For $2\leq r<k$, the left-hand side is
\begin{equation}
2(m-1)q-(r-2)q
=
q(k-r)
\leq mr(k-r),
\end{equation}
because $q=k-1=2m-1\leq2m\leq mr$. Hence
$\alpha\in\Newton(a_{\delta_k}^k)$.

\smallskip
\noindent\textbf{Part 3: vanishing coefficient.}
An asterisk on an equation number indicates that a more careful expanded derivation
is provided in Section~\ref{sec:detailed-computations}.
Write the variables as
\[
x_1,x_2,Z,
\qquad
Z=(z_1,\ldots,z_s).
\]
Let
\begin{equation}
Q(u)=\prod_{i=1}^s(u-z_i)^N.
\end{equation}
Then the Vandermonde determinant factors as
\begin{extendednumbering}
\begin{equation}\label{eq:vandermonde_factor}
    a_{\delta_N}(x_1,x_2,Z)^N
=
(x_1-x_2)^NQ(x_1)Q(x_2)a_{\delta_s}(Z)^N.
\end{equation}
\end{extendednumbering}

Define stable symmetric functions $E_r\in\Lambda_{\mathbb{Q}}$ by
\begin{equation}
    \sum_{r\geq0}E_rt^r
    =
    \exp\left(
        N\sum_{j\geq1}\frac{(-1)^{j-1}p_jt^j}{j}
    \right).
    \label{eq:stable-E}
\end{equation}
After evaluation at $Z=(z_1,\ldots,z_s)$, this identity becomes
\begin{equation}
    \sum_{r\geq0}E_r(Z)t^r
    =
    \prod_{i=1}^s(1+tz_i)^N.
    \label{eq:E-evaluation}
\end{equation}
Since $Ns=N(N-2)=(N-1)^2-1=H-1$,
we can write $Q$ in terms of $E_r$ as
\begin{extendednumbering}
\begin{equation}\label{eq:Q_as_E}
    Q(u)=\sum_{r=0}^{Ns}(-1)^rE_r(Z)u^{H-1-r}.
\end{equation}
\end{extendednumbering}

A direct coefficient extraction for $x_1,x_2$ gives
\begin{extendednumbering}
\begin{equation}\label{eq:extraction_x1_x2_coefficient}
    [x_1^Hx_2^H]
\,(x_1-x_2)^NQ(x_1)Q(x_2)
=
R(Z),
\end{equation}
\end{extendednumbering}
where
\begin{extendednumbering}
\begin{equation}
R := -\sum_{r=0}^{s}
(-1)^r\binom{N}{r+1}E_rE_{s-r}\in\Lambda_{\mathbb{Q}}.
\end{equation}
\end{extendednumbering}

Consequently,
\begin{equation}
[x^\alpha]a_{\delta_N}^N
=
[z_1^L\cdots z_s^L]
a_{\delta_s}(Z)^N R(Z).
\end{equation}

By \eqref{eq:even_power_vandermonde_as_dyson},
\begin{equation}
a_{\delta_s}(Z)^N
=
(-1)^{m\binom{s}{2}}
e_s(Z)^{m(s-1)}D_{s,m}(Z).
\end{equation}
Using \(L=m(s-1)+1\), we therefore obtain
\begin{equation}
    [x^\alpha]a_{\delta_N}^N
=
(-1)^{m\binom{s}{2}}
\CT\left(
D_{s,m}(Z)\frac{R(Z)}{e_s(Z)}
\right).
\end{equation}

Applying Lemma~\ref{lem:dyson} to the homogeneous symmetric function
$R$ of degree $s$ gives
\begin{equation}
    [x^\alpha]a_{\delta_N}^N
=
(-1)^{m\binom{s}{2}}
\CT(D_{s,m})
\frac{(-1)^s s!}{(1/m)_s}\,
\epsilon_{-1/m}(R).
\end{equation}

To compute $\epsilon_{-1/m}(R)$, first apply the specialization
coefficientwise to \eqref{eq:stable-E}:
\begin{extendednumbering}
\begin{equation}\label{eq:specialization_of_E_r_1}
\sum_{r\geq 0}\epsilon_{-1/m}(E_r)t^r
=
\exp\left(
\sum_{j\geq 1}
\frac{(-1)^{j-1}N(-1/m)t^j}{j}
\right) =(1+t)^{-2}.
\end{equation}
\end{extendednumbering}

Hence
\begin{equation}\label{eq:specialization_of_E_r_2}
\epsilon_{-1/m}(E_r)=(-1)^r(r+1).
\end{equation}

It follows that
\begin{align}
\epsilon_{-1/m}(R)
&=
-(-1)^s
\sum_{r=0}^s
(-1)^r
\binom{N}{r+1}
(r+1)(s-r+1)\\
&=
-(-1)^sN(N-1)
\sum_{r=0}^{N-2}
(-1)^r\binom{N-2}{r}\\
&=
-(-1)^sN(N-1)
(1- 1)^{N-2}\\
&=0,
\end{align}
where we used $s=N-2$ and the identity
\begin{equation}
    \binom{N}{r+1}(r+1)(N-r-1)
=
N(N-1)\binom{N-2}{r}.
\end{equation}

Since $N-2>0$, the last sum is $(1-1)^{N-2}=0$.  The constant-term
lemma now implies
\[
\CT\left(
D_{s,m}(Z)\frac{R(Z)}{e_s(Z)}
\right)=0.
\]
Therefore
\[
[x^\alpha]a_{\delta_N}^N=0.
\]
Thus $\alpha$ is a lattice point of
$\Newton(a_{\delta_N}^N)$ that is not an exponent vector
of \(a_{\delta_N}^N\). Hence \(a_{\delta_N}^N=a_{\delta_k}^k\) is non-SNP.
\end{proof}

\begin{proof}[Proof of Theorem~\ref{thm:main}]
If $k$ is odd, the three-variable fact above supplies a non-SNP seed and
\cite[Lemma~2.24]{MonicalTokcanYong2019} propagates it to every $n\geq3$.
For $k=2$, Proposition~2.23 of the same reference gives non-SNP for every
$n\geq5$.  If $k\geq4$ is even, Theorem~\ref{thm:even-seed} supplies a
seed at $n=k$, and the same propagation lemma applies.  Thus one may take
\[
    N_k=
    \begin{cases}
        3,&k\text{ odd},\\
        5,&k=2,\\
        k,&k\geq4\text{ even}.
    \end{cases}
\]
\end{proof}

\section{Proof of Lemma~\ref{lem:dyson}}

\begin{proof}
Put $\tau=1/m$, and let $P_\lambda^{(\tau)}$ denote the monic Jack
symmetric function.  For symmetric polynomials $g$ and $h$ in $s$
variables, consider the finite-variable Jack scalar product
\cite[Chapter~VI, \S10, equations~(10.34)--(10.36)]{macdonald1995symmetric}:
\begin{align}
    \langle g,h\rangle'_{s,\tau}
&=
\frac1{s!}
\operatorname{CT}\left(
g(Z)h(Z^{-1})
\prod_{i\neq j}
\left(1-\frac{z_i}{z_j}\right)^{1/\tau}
\right)\\
&=
\frac1{s!}
\operatorname{CT}\left(
g(Z)h(Z^{-1})D_{s,m}(Z)
\right).
\end{align}

The evaluated Jack polynomials $P_\lambda^{(\tau)}(Z)$, for
$\ell(\lambda)\leq s$, are orthogonal for this scalar product.  Expand
the homogeneous symmetric function $f$ in the Jack basis:
\begin{equation}
    f=\sum_{\lambda\vdash s}c_\lambda P_\lambda^{(\tau)}.
    \label{eq:jack-expansion}
\end{equation}
Every partition of $s$ has length at most $s$, so evaluation at $Z$
retains the entire expansion.  Since $P_{(1^s)}^{(\tau)}=e_s$,
orthogonality gives
\begin{equation}
   \langle f,e_s\rangle'_{s,\tau}
=
c_{(1^s)}
\langle e_s,e_s\rangle'_{s,\tau}.
\end{equation}

Since $e_s(Z^{-1})=1/e_s(Z)$,
\begin{equation}
    \langle f,e_s\rangle'_{s,\tau}
=
\frac1{s!}
\operatorname{CT}\left(
D_{s,m}(Z)\frac{f(Z)}{e_s(Z)}
\right).
\end{equation}

Also, $e_s(Z)e_s(Z^{-1})=1$, so
\begin{equation}
    \langle e_s,e_s\rangle'_{s,\tau}
=
\frac1{s!}
\operatorname{CT}\bigl(D_{s,m}(Z)\bigr).
\end{equation}

The denominator is nonzero.  Integration against normalized Haar
measure on the torus extracts the constant term of a Laurent polynomial
\cite[Chapter~VI, \S9, equation~(9.10) and the paragraph following
it]{macdonald1995symmetric}.  Thus, writing
$z_j=e^{\mathrm{i}\theta_j}$, and pairing the
factors indexed by $(i,j)$ and $(j,i)$ gives
\begin{equation}
    \CT(D_{s,m}(Z))
    =
    \frac{1}{(2\pi)^s}
    \int_{[0,2\pi]^s}
    \prod_{1\leq i<j\leq s}
    \left|1-e^{\mathrm{i}(\theta_i-\theta_j)}\right|^{2m}
    \,d\theta_1\cdots d\theta_s
    >0.
\end{equation}
The strict inequality holds because the integrand is nonnegative
everywhere and positive whenever the $e^{\mathrm{i}\theta_j}$ are
pairwise distinct.
Consequently,
\begin{equation}
\frac{
\operatorname{CT}\!\left(
D_{s,m}(Z)\frac{f(Z)}{e_s(Z)}
\right)}
{\operatorname{CT}\bigl(D_{s,m}(Z)\bigr)}
=
c_{(1^s)}.
\label{eq:ratio-is-column-coefficient}
\end{equation}

It remains to compute \(c_{(1^s)}\).  Macdonald's specialization formula
for Jack polynomials is
\cite[Chapter~VI, \S10, equation~(10.20)]{macdonald1995symmetric}:
\begin{equation}
\epsilon_{-\tau}\left(P_\lambda^{(\tau)}\right)
=
\prod_{u\in\lambda}
\frac{-\tau+\tau a'(u)-\ell'(u)}
     {\tau a(u)+\ell(u)+1},
\end{equation}
where \(a(u)\), \(\ell(u)\), \(a'(u)\), and \(\ell'(u)\) are the
arm, leg, coarm, and coleg of the cell \(u\).

If \(\lambda\neq(1^s)\), then \(\lambda\) contains the cell $u=(1,2)$,
for which $a'(u)=1$ and $\ell'(u)=0$.  The corresponding numerator is
$-\tau+\tau a'(u)-\ell'(u)=0$.  It follows that
$\epsilon_{-\tau}\left(P_\lambda^{(\tau)}\right)=0$ for every $\lambda\vdash s$ with $\lambda\neq(1^s)$.

For \(\lambda=(1^s)\), the cells are $u_i=(i,1)$ for $1\leq i\leq s$,
and
\[
a(u_i)=0,
\qquad
\ell(u_i)=s-i,
\qquad
a'(u_i)=0,
\qquad
\ell'(u_i)=i-1.
\]
Therefore,
\begin{align}
\epsilon_{-\tau}\left(P_{(1^s)}^{(\tau)}\right)
&=
\prod_{i=1}^s
\frac{-\tau-(i-1)}{s-i+1}\\
&=
(-1)^s
\frac{\tau(\tau+1)\cdots(\tau+s-1)}
     {s(s-1)\cdots1}\\
&=
(-1)^s\frac{(\tau)_s}{s!}.
\end{align}

Applying \(\epsilon_{-\tau}\) to \eqref{eq:jack-expansion} gives
\begin{align}
\epsilon_{-\tau}(f)
&=
\sum_{\lambda\vdash s}
c_\lambda
\epsilon_{-\tau}\left(P_\lambda^{(\tau)}\right)\\
&=
c_{(1^s)}
\epsilon_{-\tau}\left(P_{(1^s)}^{(\tau)}\right)\\
&=
c_{(1^s)}
(-1)^s\frac{(\tau)_s}{s!}.
\end{align}
Thus
\[
c_{(1^s)}
=
\frac{(-1)^s s!}{(\tau)_s}
\epsilon_{-\tau}(f).
\]
Since \(\tau=1/m\), this becomes
\[
c_{(1^s)}
=
\frac{(-1)^s s!}{(1/m)_s}
\epsilon_{-1/m}(f).
\]

Combining this identity with
\eqref{eq:ratio-is-column-coefficient} proves
\eqref{eq:dyson-lemma}.
\end{proof}

\section{Detailed computation}\label{sec:detailed-computations}
\subsection{Equation~\ref{eq:vandermonde_factor}: Vandermonde factorization}
Directly separating the factors involving $x_1$, $x_2$, and $Z$ gives
\begin{align}
a_{\delta_N}(x_1,x_2,Z)^N
&=
(x_1-x_2)^N
\prod_{i=1}^{s}(x_1-z_i)^N
\prod_{i=1}^{s}(x_2-z_i)^N
\left(\prod_{1\leq i<j\leq s}(z_i-z_j)\right)^N\\
&=
(x_1-x_2)^N
\left(\prod_{i=1}^{s}(x_1-z_i)^N\right)
\left(\prod_{i=1}^{s}(x_2-z_i)^N\right)
a_{\delta_s}(Z)^N\\
&=
(x_1-x_2)^N
Q(x_1)Q(x_2)a_{\delta_s}(Z)^N.
\end{align}

\subsection{Equation~\ref{eq:Q_as_E}: Expansion of Q}
By \eqref{eq:E-evaluation},
\[
\prod_{i=1}^{s}(1+t z_i)^N
=
\sum_{r=0}^{Ns}E_r(Z)t^r.
\]

We compute
\begin{align}
Q(u)
&=\prod_{i=1}^{s}(u-z_i)^N\\
&=\prod_{i=1}^{s}
   \left[u^N\left(1-\frac{z_i}{u}\right)^N\right]\\
&=u^{Ns}\prod_{i=1}^{s}
   \left(1+\left(-u^{-1}\right)z_i\right)^N.
\end{align}
Substituting \(t=-u^{-1}\) into the generating function for
\(E_r(Z)\) gives
\begin{align}
Q(u)
&=u^{Ns}\sum_{r=0}^{Ns}
   E_r(Z)\left(-u^{-1}\right)^r\\
&=u^{Ns}\sum_{r=0}^{Ns}
   (-1)^rE_r(Z)u^{-r}\\
&=\sum_{r=0}^{Ns}
   (-1)^rE_r(Z)u^{Ns-r}.
\end{align}
Finally, since \(Ns=H-1\), we obtain
$
Q(u)
=
\sum_{r=0}^{Ns}
(-1)^rE_r(Z)u^{H-1-r}
$.

\subsection{Equation~\ref{eq:extraction_x1_x2_coefficient}: Coefficient extraction}
Expanding term by term gives
\begin{align}
&
[x_1^H x_2^H]
\left(
\sum_{p=0}^{N}
(-1)^p\binom{N}{p}
x_1^p x_2^{N-p}
\right)
Q(x_1)Q(x_2)\\
&=
\sum_{p=0}^{N}
(-1)^p\binom{N}{p}
[x_1^H x_2^H]\,
x_1^p x_2^{N-p}Q(x_1)Q(x_2)\\
&=
\sum_{p=0}^{N}
(-1)^p\binom{N}{p}
[x_1^{H-p}]Q(x_1)\,
[x_2^{H-(N-p)}]Q(x_2)\\
&=
\sum_{p=1}^{N-1}
(-1)^p\binom{N}{p}
[u^{H-p}]Q(u)\,
[u^{H-N+p}]Q(u),
\end{align}
The terms $p=0$ and $p=N$ vanish because $\deg Q=H-1$.

From \eqref{eq:Q_as_E},
\[
Q(u)
=
\sum_{a=0}^{Ns}
(-1)^aE_a(Z)u^{H-1-a}.
\]
Thus
\begin{equation}
    [u^{H-p}]Q(u)
=
(-1)^{p-1}E_{p-1}(Z),\qquad
    [u^{H-N+p}]Q(u)
=
(-1)^{N-p-1}E_{N-p-1}(Z).
\end{equation}

Substituting these two coefficients gives
\begin{align}
&\sum_{p=1}^{N-1}
(-1)^p\binom{N}{p}
(-1)^{p-1}E_{p-1}(Z)
(-1)^{N-p-1}E_{N-p-1}(Z)\\
&=
\sum_{p=1}^{N-1}
(-1)^p\binom{N}{p}
(-1)^{N-2}
E_{p-1}(Z)E_{N-p-1}(Z)\\
&=
\sum_{p=1}^{N-1}
(-1)^p\binom{N}{p}
E_{p-1}(Z)E_{N-p-1}(Z)\\
&=
-\sum_{r=0}^{s}
(-1)^r
\binom{N}{r+1}
E_r(Z)E_{s-r}(Z) =
R(Z).
\end{align}

\subsection{Equation~\ref{eq:specialization_of_E_r_1}: Specialization of E}
Extend $\epsilon_{-1/m}$ coefficientwise to
$\Lambda_{\mathbb{Q}}[[t]]$.  Applying it to the stable identity
\eqref{eq:stable-E}, and using $N/m=2$, gives
\begin{align}
\sum_{r\geq0}\epsilon_{-1/m}(E_r)t^r
&=
\exp\left(
N\sum_{j\geq1}
\frac{(-1)^{j-1}}{j}\epsilon_{-1/m}(p_j)t^j
\right)\\
&=
\exp\left(
-\frac{N}{m}
\sum_{j\geq1}\frac{(-1)^{j-1}}{j}t^j
\right)\\
&=
\exp\left(-2\log(1+t)\right)
=(1+t)^{-2}.
\end{align}

Finally, the negative-binomial expansion gives
\begin{align}
(1+t)^{-2}
&=
\sum_{r\geq0}\binom{-2}{r}t^r\\
&=
\sum_{r\geq0}(-1)^r(r+1)t^r.
\end{align}
Comparing coefficients of \(t^r\) yields
\eqref{eq:specialization_of_E_r_2}:
\begin{equation}
    \epsilon_{-1/m}(E_r)
=
(-1)^r(r+1).
\end{equation}

\section{Generative AI usage}
OpenAI Codex (GPT Sol 5.6 Extra High) was used as an interactive research and
writing aid in developing this proof.  All prompts and follow-up
instructions to Codex were supplied by the first author.  In the initial
conversation (reproduced below), Codex assisted with the search for a
uniform construction, computational testing of candidate exponent vectors,
and the derivation and verification of the Dyson--Jack constant-term
argument.
The generation process took approximately 5.5--6 hours; the exact duration
is unknown because the session crashed several times.  The conversation
took place on a November 2023 MacBook Pro with an Apple M3 Pro chip and
36 GB of memory.

Subsequently, Codex was used to clarify intermediate identities, expand
omitted coefficient calculations, identify relevant references, and
prepare the argument in \LaTeX.  The authors reviewed the resulting
proof and citations, assembled the paper, and take full responsibility
for the final mathematical claims.

The accompanying Lean formalization is available at\newline
\url{https://github.com/steven-le-thien/vandermonde-snp}.

\section{Acknowledgments}
The first author thanks Alexander Yong for his introduction to this
problem and subsequent discussions.  The project is sponsored by the
Defense Advanced Research Projects Agency under cooperative agreement
HR0011262E027.  The content of the information does not necessarily
reflect the position or the policy of the Government, and no official
endorsement should be inferred.

\bibliographystyle{alpha}
\bibliography{references}
\appendix

\section{Codex conversation and initial prompt}
We provide below the full conversation that generated the proof, followed
by the initial prompt.  Codex generated that prompt after being asked to
follow and mimic the prompt used by OpenAI to resolve the cycle double
cover conjecture.

\subsection{Conversation}
\begin{Verbatim}
# Prove Vandermonde power non-SNP

> User context
>
> Attachments:
> - You are directing a long-horizon mathematical research program whose sole objec…: `~/.codex/attachments/9cdac09c-1075-4e8a-a36a-5b938d707337/pasted-text.txt`

<details><summary>12 previous messages</summary>

> I’ll read the attached research brief first, then carry out the requested work and verify any deliverables it calls for.
>
> I’ve parsed the brief. The odd-power argument is concrete; the hard part is a uniform even-power hole. I’m setting this up as a persistent research goal, beginning with three independent tracks in parallel: exact small-dimension coefficient identities, Dyson/constant-term methods, and computational pattern discovery plus adversarial testing. I’ll also update the literature search independently before trusting any claimed resolution.
>
> <details><summary>Searched the web</summary>
>
> <details><summary>Searched the web</summary>
>
> - Searched the web for `"Conjecture 2.25" "Newton Polytopes in Algebraic Combinatorics" Vandermonde SNP ...`
> - Searched the web for `site:arxiv.org/abs/1201.4572 Ballantine Vandermonde ...`
> - Searched the web for `"a_{delta_n}" "SNP" Vandermonde ...`
>
> </details>
>
> </details>
>
> The first structural audit has already ruled out a tempting shortcut: equal-variable specialization forces many coefficient relations, but its easiest isolated zeros lie below the permutahedron’s lower-face bound. I’m now testing near-face and centered coefficients, where Dyson/Jack identities can produce genuine characteristic-zero cancellation without violating majorization.
>
> <details><summary>Created a file, edited a file, searched the web</summary>
>
> - Wrote `./work/near_vertex.py` (+87 -0)
>
> - Wrote `./work/near_vertex.py` (+34 -0)
>
> <details><summary>Searched the web</summary>
>
> - Searched the web for `Jack polynomial Vandermonde power monomial coefficient zero negative parameter staircase ...`
> - Searched the web for `"Determinantal Expression and Recursion for Jack Polynomials" pdf Lapointe Lascoux Morse`
>
> </details>
>
> </details>
>
> The literature trail points to a sharper route through the Jack-polynomial recurrence: at the Vandermonde parameter \(\alpha=-2/(2m-1)\), certain recurrence factors can vanish exactly. I’m using the PDF inspection workflow now to verify the original determinant/recurrence statements and avoid importing a normalization or specialization outside its hypotheses.
>
> <details><summary>Searched the web</summary>
>
> <details><summary>Searched the web</summary>
>
> - Searched the web for `Knop Sahi combinatorial formula Jack polynomials admissible tableaux coefficient weight formula critical boxes ...`
>
> </details>
>
> </details>
>
> The first round has now produced exact certificates for the published cases and one genuine infinite family of even powers, but it does not cover all even \(k\). I’m deliberately not promoting it to a proof. A second independent round is starting around the strongest new lead: multivariate cut-slack coefficients reduce to Krawtchouk-type finite differences, so the target is an identity whose zero occurs for every even \(k\), with the majorization inequalities built in rather than checked afterward.
>
> <details><summary>Edited a file, searched the web</summary>
>
> - Wrote `./work/near_vertex.py` (+2 -2)
>
> <details><summary>Searched the web</summary>
>
> - Searched the web for `"Conjecture 2.25" Vandermonde Monical Tokcan Yong ...`
> - Searched the web for `inverse Jack Kostka coefficient rectangle monomial Jack polynomial zero ...`
> - Searched the web for `Tewodros Amdeberhan Vandermonde powers Newton polytope Jack 2024 ...`
>
> </details>
>
> </details>
>
> Context compaction
>
> Source: automatic
> Status: completed
>
> The literature audit confirms there is no known proof to import; the even case really is the open core. I’m now pressing on constructions whose dimension can grow with \(k\), since exhaustive work has ruled out every fixed-\(n=5\) affine template.
>
> A useful new angle is emerging from the near-vertex expansion: fixed small cut slacks turn the target coefficient into an explicit integer polynomial in \(k\). I’m testing whether a tunable local “gadget” can give the factor \(k-2m\) for arbitrary \(m\); that would settle the even case cleanly and propagate automatically.
</details>

> my computer got disconnected, continue with the previous prompt

<details><summary>17 previous messages</summary>

> Continuing from the exact point of interruption. The literature audit, odd-power proof, finite certificates, and current even-power constructions are preserved; I’m resuming the orientation/Eulerian-subgraph attack on the universal even case.
>
> The resumed proof search has reached a sharper formulation: a target coefficient is the signed imbalance of Eulerian subdigraphs of a chosen orientation of \(kK_n\). I’m combining that with a two-variable “thin slice” of the Vandermonde, where the residual factor is an explicit Hahn/Jacobi polynomial; the goal is a parity-forced zero at an admissible weight, not merely a numerical cancellation.
>
> <details><summary>Edited a file, searched the web</summary>
>
> - Wrote `./work/near_vertex.py` (+9 -0)
>
> - Wrote `./work/near_vertex.py` (+35 -0)
>
> <details><summary>Searched the web</summary>
>
> - Searched the web for `"powers of the Vandermonde" zero coefficients monomial expansion ...`
>
> </details>
>
> </details>
>
> <details><summary>Searched the web</summary>
>
> <details><summary>Searched the web</summary>
>
> - Searched the web for `Jack characters rectangular Young diagrams explicit formula zeros Stanley character formula ...`
> - Searched the web for `Jack polynomial negative parameter coefficient sign rule squeezing Laughlin coefficients signs ...`
>
> </details>
>
> </details>
>
> Context compaction
>
> Source: automatic
> Status: completed
>
> I’ve resumed from the exact search state, including the verified odd-power proof, propagation lemma, finite even-power certificates, and the outstanding uniform even-power case. I’m checking the surviving parallel searches now, then I’ll push the two most promising structural routes: diagonal-fiber cancellation and rectangular Jack/Dyson coefficients.
>
> A genuinely uniform candidate has just emerged from the Dyson/Schur route: for \(k=2m\), it proposes \(n=4m\) and a two-level exponent vector, reducing the missing coefficient to one specialized generalized Dyson constant term. I’m treating this as a theorem-strength lead, not a result yet—the decisive task is to reconstruct the needed vanishing lemma self-containedly and audit the indexing, specialization, and polytope membership independently.
>
> <details><summary>Searched the web</summary>
>
> <details><summary>Searched the web</summary>
>
> - Searched the web for `Huang Jiang Zhou generalized q-Dyson Schur vanishing theorem 2026 Kadell orthogonality ...`
>
> </details>
>
> </details>
>
> The first uniform specialization failed an adversarial check at \(m=1\): the cited two-part product was not the ordinary Dyson product. That route is rejected. There is, however, a potentially reparable specialization—taking the distinguished block to have size \(n-1\) makes every pair receive the same decrement and recovers the ordinary Dyson product exactly—so I’m now deriving the theorem’s actual vanishing inequalities before deciding whether it yields an interior coefficient.
>
> I’m shifting to a more elementary invariant that could eliminate the constant-term machinery entirely. Translation invariance gives
> \[
> \sum_i(\gamma_i+1)[x^{\gamma+e_i}]\Delta_n^k=0.
> \]
> The key is to search for a degree-\(D-1\) partition \(\gamma\) having exactly one increment-orbit inside \(kP_{\delta_n}\); then symmetry forces that orbit’s coefficient to vanish. I’m testing this criterion symbolically across \(k,n\) to see whether it has a uniform facet pattern.
>
> <details><summary>Created a file</summary>
>
> - Wrote `./work/search_translation_cert.py` (+79 -0)
>
> </details>
>
> The diagonal-fiber method has now been ruled out in its naive form: every two-coordinate section of the permutahedron contains at least \(k+1\) lattice points, so a coalescence identity cannot isolate one orbit when \(k\ge2\). I’m retaining confluent identities but, in parallel, examining a network-flow gadget where the number of intermediate vertices depends on \(k\); two-unit flows are the first place where direct and routed contributions can acquire a factor of the form \(k-L\).
>
> <details><summary>Created a file, searched the web</summary>
>
> <details><summary>Searched the web</summary>
>
> - Searched the web for `"Conjecture 2.25" "Vandermonde" SNP ...`
> - Searched the web for `integer point permutahedron sum of permutations vertex semigroup normality permutohedron ...`
>
> </details>
>
> - Wrote `./work/group_sum_decompose.py` (+61 -0)
>
> </details>
>
> <details><summary>Searched the web</summary>
>
> <details><summary>Searched the web</summary>
>
> - Searched the web for `OEIS number of vectorial sums of three permutations ...`
>
> </details>
>
> </details>
>
> A new differential route looks substantially stronger. The reflection-group identity
> \[
> \Delta(\partial)\Delta^{k+1}=C_{n,k}\Delta^k\qquad(C_{n,k}\ne0)
> \]
> reduces a desired even coefficient to coefficients of the alternating odd power. Thus it is enough to find \(\alpha\in kP_\delta\) such that, for every permutation \(p\) of \(0,\ldots,n-1\), the vector \(\alpha+p\) has a repeated coordinate. Every contributing odd-power coefficient then vanishes individually. I’m now searching this finite “no distinct shifted transversal” condition; if it admits a uniform construction, it gives a short self-contained proof.
>
> The shifted-transversal variant collapses for a simple reason: sorting \(\alpha\) and assigning \(0,1,\ldots,n-1\) in the same order always makes the shifted entries distinct. I’ve rejected it.
>
> A stronger use of the same differential invariance survives: every degree-\(D-1\) partition with exactly two admissible unit increments gives a positive two-term relation between two coefficients, forcing opposite signs if both are nonzero. An odd cycle of such relations forces at least one coefficient to vanish. This only needs an existential hole and may admit a uniform boundary-cycle construction, so I’m building the exact relation graph now.
>
> <details><summary>Created a file, edited a file</summary>
>
> - Wrote `./work/translation_graph.py` (+83 -0)
> - Wrote `./work/translation_graph.py` (+37 -0)
>
> </details>
</details>

> codex crashed, resume previous prompt

<details><summary>27 previous messages</summary>

> Resumed again from the saved research state. The odd case and propagation remain closed; the uniform even case is still the only theorem-strength gap. Immediately before the crash I was testing whether translation-invariance relations form a weighted cycle whose coefficient ratios are inconsistent—such a cycle would force a hole without needing an explicit coefficient formula. I’m restoring that computation and checking the parallel searches now.
>
> The two-term relation graphs are acyclic in every tested case, so the weighted-cycle certificate does not materialize. I’m testing a sharper symmetry variant instead: if the two endpoints of one translation relation are complementary under
> \[
> c_\alpha=c_{\,k(n-1)\mathbf1-\alpha},
> \]
> then the positive two-term relation immediately forces that coefficient to be zero. This converts the search from a global cycle to a single exact boundary configuration.
>
> I’m revisiting Kadell orthogonality in a form that avoids the failed two-part specialization. The standard identity says
> \[
> \operatorname{CT}\bigl(x^{-v}h_r[mX]D_{n,m}\bigr)=0
> \quad\text{when }|v|=r,\ \max v_i<r.
> \]
> If the intersection of \(v-\operatorname{supp}h_r[mX]\) with the Dyson polytope consists of a single permutation orbit, every term is a positive multiple of one Dyson coefficient, forcing it to vanish. This is now a finite convex-intersection search for a pattern uniform in \(m\).
>
> <details><summary>Created a file, edited a file</summary>
>
> - Wrote `./work/search_kadell_orbit.py` (+58 -0)
> - Wrote `./work/search_kadell_orbit.py` (+31 -1)
>
> </details>
>
> The translation system itself is now fully understood: symmetric translation invariance plus the pair facets uniquely characterize \(\Delta_n^k\), but that uniqueness does not yet produce a uniform zero.
>
> I’m testing a projection that may: under \(x_i=t^{i-1}\),
> \[
> \Delta_n(t)^{k}=t^{C}\prod_{d=1}^{n-1}(1-t^d)^{k(n-d)}.
> \]
> If a zero univariate coefficient occurs at a weight level whose lattice points form a single permutation orbit, symmetry forces the underlying multivariate coefficient to be zero. Both the projected coefficient and the orbit fiber can be computed exactly.
>
> <details><summary>Created a file, searched the web</summary>
>
> - Wrote `./work/principal_orbit_search.py` (+54 -0)
>
> <details><summary>Searched the web</summary>
>
> - Searched the web for `rectangular Jack character cycle type equal parts formula rectangle Stanley Lassalle ...`
> - Searched the web for `Dyson constant term coefficient (a repeated r, -a repeated r) formula ...`
>
> </details>
>
> </details>
>
> Context compaction
>
> Source: automatic
> Status: completed
>
> I’m resuming from the exact research state, including the active proof audits. The latest projection and local-involution routes are now closed; I’m testing the strongest remaining structural lead—a balanced two-block family that already vanishes in the first nontrivial cases—while also checking whether Morris/Aomoto constant-term vanishing can isolate the needed single coefficient.
>
> <details><summary>Searched the web</summary>
>
> <details><summary>Searched the web</summary>
>
> - Searched the web for `"A Unified Elementary Approach" Dyson Morris Aomoto constant term ...`
>
> </details>
>
> </details>
>
> The first balanced guess has been decisively falsified at \(m=3\) by two independent exact computations. A stronger family has emerged: for \(k=2m\), take \(n=4m+2\) and the centered exponent \((2m^{\,2m+1},-2m^{\,2m+1})\); it vanishes for \(m=1,2\). I’m now reducing this family to a rectangular Jack coefficient and separately testing whether its zero is combinatorial (no permutation decomposition) or genuinely cancellation-based.
>
> <details><summary>Searched the web</summary>
>
> <details><summary>Searched the web</summary>
>
> - Searched the web for `rectangular Jack polynomial character Stanley formula exact rectangle ...`
>
> </details>
>
> </details>
>
> <details><summary>Searched the web</summary>
>
> <details><summary>Searched the web</summary>
>
> - Searched the web for `rectangular Jack polynomial hyperdeterminant alpha integer formula ...`
> - Searched the web for `Jack polynomial specialization at roots of unity rectangular partition ...`
>
> </details>
>
> </details>
>
> A uniform even-power certificate has just crystallized. The rectangular Jack evaluation admits a second, “dual” hyperdeterminant formula whose entries are elementary symmetric functions of the root alphabet. At \(n=4m+2\), every potentially nonzero tensor entry would force a negative row imbalance, while the total imbalance of permutations is zero—so every hyperdeterminant summand vanishes. I’m now checking the parameter normalization and index ranges line by line and running adversarial small-case tests before treating this as a proof.
>
> The apparent uniform construction failed a source-level audit: the two-part product’s reverse cross exponent depends on the first index, so the midpoint specialization is missing one cross factor and is not the Dyson product. I’ve rejected that proof. The same theorem has only one genuinely nontrivial uniform specialization (\(n_0=n-1\)); I’m now testing whether its Schur relations can be combined or triangularly inverted to isolate one Dyson coefficient.
>
> <details><summary>Created a file</summary>
>
> - Wrote `./work/search_hjz_uniform.py` (+117 -0)
>
> </details>
>
> <details><summary>Searched the web</summary>
>
> <details><summary>Searched the web</summary>
>
> - Searched the web for `multi-component q-Dyson constant term Jiang Wang Xin Zhou two part uniform exponents ...`
>
> </details>
>
> </details>
>
> Context compaction
>
> Source: automatic
> Status: completed
>
> I’ve restored the proof state and am continuing from the last viable frontier: generalized \(q\)-Dyson vanishing relations, exact Schur-relation computations, and the remaining even-exponent case.
>
> A new structural reduction is emerging: the even case is equivalent to finding a zero inverse Jack transition coefficient for a rectangular partition. I’m checking whether the known rectangular Jack-character formulas supply a factor that vanishes for every integral parameter \(m\), while the other agents extract a direct certificate from the newer \(q\)-Dyson relations.
>
> <details><summary>Searched the web</summary>
>
> <details><summary>Searched the web</summary>
>
> - Searched the web for `site:arxiv.org rectangular Jack polynomial monomial coefficient inverse Kostka zero ...`
> - Searched the web for `"Conjecture 2.25" "Newton Polytopes in Algebraic Combinatorics" ...`
>
> </details>
>
> </details>
>
> <details><summary>Searched the web</summary>
>
> <details><summary>Searched the web</summary>
>
> - Searched the web for `Dyson product zero coefficients root lattice exact zeros ...`
>
> </details>
>
> </details>
>
> <details><summary>Created a file, searched the web</summary>
>
> - Wrote `./work/target_slack_dp.cpp` (+54 -0)
>
> <details><summary>Searched the web</summary>
>
> - Searched the web for `Laughlin wavefunction monomial coefficients vanish missing squeezed partitions ...`
> - Searched the web for `Laughlin state "zero coefficients" squeezed partitions ...`
> - Searched the web for `AFLT q-Dyson constant term identity Macdonald P lambda P mu ...`
>
> </details>
>
> </details>
>
> Context compaction
>
> Source: automatic
> Status: completed
>
> I’ve recovered the active research state and the interrupted computations. I’m now checking the pending symbolic family and the two remaining structural searches; the earlier AFLT shortcut has been definitively ruled out because its specialization parameters are positive, not negative.
>
> <details><summary>Searched the web</summary>
>
> <details><summary>Searched the web</summary>
>
> - Searched the web for `site:arxiv.org rectangular Jack character Stanley formula rectangle p q permutations formula ...`
>
> </details>
>
> </details>
>
> The collision and isolated-fiber route is now closed: total unimodularity shows every identified-variable fiber contains several distinct coefficient orbits, and the full derivative system only restates Vandermonde divisibility. I’m shifting the main search to exact signed-flow identities and to a stronger “vertex-sum obstruction,” where vanishing would follow before any cancellation is needed.
>
> <details><summary>Created a file, edited a file</summary>
>
> - Wrote `./work/search_n4_small_ends.py` (+44 -0)
> - Wrote `./work/search_n4_small_ends.py` (+4 -0)
>
> </details>
</details>

<details><summary>6 previous messages</summary>

> <details><summary>Searched the web</summary>
>
> - Searched the web for `site:arxiv.org inverse Kostka Jack coefficient rectangular partition monomial rectangle ...`
>
> </details>
>
> A useful reduction has emerged, though it is not yet the proof: root-of-unity clustering turns \(\Delta_{mn}^2\) into a monomial times \(\Delta_n^{2m}\). This recasts the even case as a two-permutation block-sum problem. The naive scaled \(k=2\) hole survives for \(k=4\) but fails at \(k=6\), so I’m testing sentinel versions that preserve the two-permutation obstruction without relying on scaling.
>
> <details><summary>Searched the web</summary>
>
> <details><summary>Searched the web</summary>
>
> - Searched the web for `Borda score vectors sum of k permutations characterization majorization non realizable ...`
>
> </details>
>
> </details>
>
> try not to use more than 100GB of application memory at a time
>
> Understood. I’ll keep the searches memory-bounded and well below 100 GB—favoring symbolic recurrences, sparse state spaces, and streamed checks over large full-support expansions.
>
> <details><summary>Created a file, edited a file, searched the web</summary>
>
> <details><summary>Searched the web</summary>
>
> - Searched the web for `Jack polynomial rectangular specialization roots of unity e_L coefficient hyperdeterminant ...`
>
> </details>
>
> - Wrote `./work/search_twoblock_targets.py` (+52 -0)
>
> - Wrote `./work/search_twoblock_targets.py` (+1 -1)
>
> <details><summary>Searched the web</summary>
>
> - Searched the web for `rectangular Jack character two cycles explicit formula ...`
>
> </details>
>
> </details>
</details>

<details><summary>10 previous messages</summary>

> I’ve found a concrete uniform even-power candidate, verified exactly modulo two large primes for \(m=2,3,4\): with \(n=2m\), take
> \[
> \alpha=\bigl((2m-1)^2,(2m-1)^2,((m-1)(2m-1))^{\,2m-2}\bigr).
> \]
> The remaining task is now sharply localized: prove one rectangular Jack-character coefficient at cycle type \((m(2m-1),m(2m-1))\) vanishes identically. I’m reconstructing that identity from the original recurrence and will then audit the Dyson/monomial conversion.
>
> try not to use more than 80GB of application memory at a time
>
> Understood—I’ll stay below 80 GB. The current proof work is symbolic and sparse; I’ll avoid any full-support expansion whose state estimate approaches that limit.
>
> Context compaction
>
> Source: automatic
> Status: completed
>
> I’m resuming the proof search from the recovered state and will keep application memory below 80 GB. The active bottleneck is now very narrow: proving the rectangular Jack-character vanishing that converts the verified even-exponent construction into a uniform proof.
>
> The even case has now yielded a uniform characteristic-zero cancellation, not just computational evidence. After extracting the two exceptional variables, the target coefficient becomes a Dyson constant term of a degree-\(s\) symmetric polynomial; a Jack-duality lemma reduces that constant term to the specialization \(p_j\mapsto-1/m\), where the remaining sum collapses to \(2m(2m-1)(1-1)^{2m-2}=0\). I’m checking every normalization, sign, and majorization inequality before assembling the final proof.
>
> <details><summary>Created a file</summary>
>
> - Wrote `./work/final_proof_draft.md` (+280 -0)
>
> </details>
>
> <details><summary>Edited a file</summary>
>
> - Wrote `./work/final_proof_draft.md` (+31 -7)
> - Wrote `./work/final_proof_draft.md` (+12 -5)
>
> </details>
>
> The integrated proof has passed two independent line-by-line audits, including exact checks at the smallest boundary cases \(k=1,2,4\). I’ve also hardened the only presentation-sensitive point—the finite-variable Jack specialization—by making the stable lift, parameter convention, duality normalization, and positivity of the Dyson denominator explicit. One final adversarial review is still checking for any hidden theorem-strength gap.
>
> <details><summary>Edited a file</summary>
>
> - Wrote `./work/final_proof_draft.md` (+9 -3)
>
> </details>
</details>

## Proof

Let
\[
V_n(x)=a_{\delta_n}(x)=\prod_{1\le i<j\le n}(x_i-x_j),
\qquad
\delta_n=(n-1,n-2,\ldots,0).
\]

We construct, for every \(k\ge1\), an integral exponent vector in
\(\operatorname{Newt}(V_n^k)\) whose coefficient is zero.

The alternant formula
\[
V_n=\det(x_i^{\,n-j})_{i,j=1}^n
\]
shows that its exponent vectors are precisely the permutations of \(\delta_n\). Hence
\[
\operatorname{Newt}(V_n)=P_{\delta_n}.
\]
Newton polytopes add under multiplication: for every linear functional, the corresponding initial form of a product is the nonzero product of the initial forms. Therefore
\[
\operatorname{Newt}(V_n^k)=kP_{\delta_n}. \tag{1}
\]

We use the majorization description
\[
a\in kP_{\delta_n}
\iff
\sum_{i=1}^n a_i=\frac{kn(n-1)}2
\quad\text{and}\quad
\sum_{i=1}^r a_i^\downarrow
\le k\left(rn-\frac{r(r+1)}2\right)
\tag{2}
\]
for \(1\le r<n\).

### A Dyson constant-term lemma

For \(s,m\ge1\), put
\[
D_{s,m}(z)=\prod_{i\ne j}(1-z_i/z_j)^m,
\qquad e_s=z_1\cdots z_s,
\]
and let \(\operatorname{CT}\) denote constant term. For a scalar \(c\), let
\(\epsilon_c\) be the specialization of symmetric functions determined by
\[
\epsilon_c(p_j)=c\qquad(j\ge1).
\]
A symmetric polynomial of degree at most \(s\) in \(s\) variables has a unique stable lift in the same degree, so this specialization is unambiguous below.

**Lemma.** If \(f\) is symmetric and homogeneous of degree \(s\), then
\[
\frac{\operatorname{CT}\!\left(D_{s,m}f/e_s\right)}
     {\operatorname{CT}(D_{s,m})}
=
\frac{(-1)^s s!}{(1/m)_s}\,\epsilon_{-1/m}(f), \tag{3}
\]
where \((a)_s=a(a+1)\cdots(a+s-1)\).

**Proof.**
Set \(\alpha=1/m\). Let \(P_\lambda^{(\alpha)}\) be the monic Jack polynomials and \(Q_\lambda^{(\alpha)}\) their duals for
\[
\langle p_\lambda,p_\mu\rangle_\alpha
=\delta_{\lambda\mu}z_\lambda\alpha^{\ell(\lambda)}.
\]
The finite-torus Jack scalar product is
\[
\langle f,g\rangle'_{s,\alpha}
=
\frac1{s!}\operatorname{CT}\left(
f(z)g(z^{-1})
\prod_{i\ne j}(1-z_i/z_j)^{1/\alpha}
\right). \tag{4}
\]
The Jack polynomials in at most \(s\) variables are orthogonal for (4). Algebraically, this follows by applying
\(\operatorname{CT}(z_i\partial_{z_i}h)=0\) to the commuting Jack operators: these operators are self-adjoint for (4), and their monic triangular joint eigenvectors are the \(P_\lambda^{(\alpha)}\).

Because
\[
P_{(1^s)}^{(\alpha)}=e_s,
\]
write
\[
f=\sum_{\lambda\vdash s}a_\lambda P_\lambda^{(\alpha)}.
\]
Orthogonality and \(e_s(z)e_s(z^{-1})=1\) imply
\[
\frac{\operatorname{CT}(D_{s,m}f/e_s)}
     {\operatorname{CT}(D_{s,m})}
=a_{(1^s)}. \tag{5}
\]

Define the Jack involution by
\[
\omega_\alpha(p_j)=(-1)^{j-1}\alpha p_j.
\]
Jack duality gives
\[
\omega_\alpha P_\lambda^{(\alpha)}
=Q_{\lambda'}^{(1/\alpha)}. \tag{6}
\]
This follows because \(\omega_\alpha\) interchanges the power-sum scalar products with parameters \(\alpha\) and \(1/\alpha\), conjugates dominance triangularity, and preserves the required dual normalization.

Apply \(\omega_\alpha\) to the Jack expansion of \(f\), then specialize to the one-letter alphabet \(1\). Every \(Q_{\lambda'}(1)\) vanishes unless \(\lambda'=(s)\), equivalently \(\lambda=(1^s)\). Moreover, the one-row part of the Jack reproducing kernel gives
\[
\sum_{r\ge0}Q_{(r)}^{(\beta)}t^r
=
\exp\left(\frac1\beta
\sum_{j\ge1}\frac{p_jt^j}{j}\right).
\]
Thus
\[
Q_{(s)}^{(1/\alpha)}(1)
=[t^s](1-t)^{-\alpha}
=\frac{(\alpha)_s}{s!}. \tag{7}
\]

For every power-sum monomial \(p_\mu\) of degree \(s\),
\[
\epsilon_1(\omega_\alpha p_\mu)
=(-1)^{s-\ell(\mu)}\alpha^{\ell(\mu)}
=(-1)^s\epsilon_{-\alpha}(p_\mu).
\]
Consequently,
\[
(-1)^s\epsilon_{-\alpha}(f)
=a_{(1^s)}\frac{(\alpha)_s}{s!}.
\]
Together with (5), this proves (3).

Finally, its denominator is nonzero: on \(z_i=e^{i\theta_i}\),
\[
D_{s,m}(z)=\prod_{i<j}|1-e^{i(\theta_i-\theta_j)}|^{2m},
\]
so
\[
\operatorname{CT}(D_{s,m})
=
\frac1{(2\pi)^s}
\int_{[0,2\pi]^s}
\prod_{i<j}|1-e^{i(\theta_i-\theta_j)}|^{2m}\,d\theta>0.
\]
\(\square\)

### Odd \(k\)

Suppose \(k\) is odd and take \(n=3\) and
\[
a_k=\left(1,\frac{3k-1}{2},\frac{3k-1}{2}\right).
\]
It is integral and has total sum \(3k\). In decreasing order,
\[
\frac{3k-1}{2}\le2k,
\qquad
3k-1\le3k,
\]
so (2) gives \(a_k\in kP_{\delta_3}\).

Since \(k\) is odd, \(V_3^k\) is alternating. Interchanging the last two variables fixes \(x^{a_k}\) but negates \(V_3^k\). Hence
\[
[x^{a_k}]V_3^k=0. \tag{8}
\]

### The case \(k=2\)

Take \(n=6\) and
\[
a_2=(7,7,7,3,3,3).
\]
Its decreasing partial sums are
\[
7,\ 14,\ 21,\ 24,\ 27,
\]
while those of \(2\delta_6=(10,8,6,4,2,0)\) are
\[
10,\ 18,\ 24,\ 28,\ 30.
\]
The totals are both \(30\), so \(a_2\in2P_{\delta_6}\).

Each monomial of \(V_6\) has an exponent vector which is a permutation of
\((0,1,2,3,4,5)\). If \(x^{a_2}\) occurred in \(V_6^2\), there would be two such permutations whose coordinatewise sum was \(a_2\).

At each coordinate having sum \(3\), the two entries must form either
\(\{0,3\}\) or \(\{1,2\}\). Coordinates having sum \(7\) contain neither \(0\) nor \(1\). Across the two permutations there are two occurrences of \(0\); hence exactly two of the three low coordinates must have type \(\{0,3\}\). Only one low coordinate then has type \(\{1,2\}\), yielding only one occurrence of \(1\), although the two permutations contain two occurrences of \(1\). This is impossible. Thus
\[
[x^{a_2}]V_6^2=0. \tag{9}
\]

### Even \(k\ge4\)

Write \(k=2m\), where \(m\ge2\), and set
\[
N=2m,\qquad s=N-2,\qquad q=N-1,
\]
\[
H=q^2,\qquad L=(m-1)q=m(s-1)+1. \tag{10}
\]
In \(N\) variables define
\[
a_{2m}=(H,H,\underbrace{L,\ldots,L}_{s\text{ times}}). \tag{11}
\]

Its total sum is
\[
2q^2+(2m-2)(m-1)q
=2m^2q
=\frac{(2m)N(N-1)}2.
\]
Subtract the common mean \(mq\) from every coordinate. The decreasing vector becomes
\[
b=((m-1)q,(m-1)q,\underbrace{-q,\ldots,-q}_{N-2}). \tag{12}
\]
After the same centering, the majorization inequalities (2) become
\[
\sum_{i=1}^r b_i\le mr(N-r). \tag{13}
\]
For \(r=1\),
\[
(m-1)q\le mq.
\]
For \(2\le r<N\),
\[
\sum_{i=1}^r b_i
=2(m-1)q-(r-2)q
=q(2m-r)
\le mr(2m-r),
\]
because \(q=2m-1\le mr\). Hence
\[
a_{2m}\in2mP_{\delta_N}. \tag{14}
\]

We now prove its coefficient is zero. Write the variables as
\[
x_1,x_2,Z,\qquad Z=(z_1,\ldots,z_s),
\]
and define
\[
E_r=e_r[NZ],
\qquad
\sum_{r\ge0}E_rt^r=\prod_{i=1}^s(1+t z_i)^N. \tag{15}
\]
Thus \(E_r\) is the elementary symmetric polynomial in the multiset containing \(N\) copies of every \(z_i\).

Put
\[
Q(u)=\prod_{i=1}^s(u-z_i)^N.
\]
Since
\[
Ns=N(N-2)=(N-1)^2-1=H-1,
\]
we have
\[
Q(u)=\sum_{r=0}^{Ns}(-1)^rE_r\,u^{Ns-r}. \tag{16}
\]
Also,
\[
V_N^N
=
V_s(Z)^N(x_1-x_2)^NQ(x_1)Q(x_2). \tag{17}
\]

Expand
\[
(x_1-x_2)^N
=
\sum_{j=0}^N(-1)^j\binom Njx_1^{N-j}x_2^j.
\]
The endpoint terms cannot contribute to \(x_1^Hx_2^H\), because \(\deg Q=H-1\). For \(1\le j\le N-1\), the required indices in the two copies of \(Q\) are
\[
N-j-1,\qquad j-1.
\]
Their signs from (16) multiply to
\[
(-1)^{N-2}=1.
\]
Therefore
\[
[x_1^Hx_2^H](x_1-x_2)^NQ(x_1)Q(x_2)=R(Z), \tag{18}
\]
where
\[
R(Z)
=
-\sum_{r=0}^s(-1)^r\binom N{r+1}E_rE_{s-r}. \tag{19}
\]
The polynomial \(R\) is symmetric and homogeneous of degree \(s\). By (17),
\[
[x^{a_{2m}}]V_N^N
=
[z_1^L\cdots z_s^L]\,V_s(Z)^N R(Z). \tag{20}
\]

For every \(i<j\),
\[
(1-z_i/z_j)^m(1-z_j/z_i)^m
=
(-1)^m\frac{(z_i-z_j)^{2m}}{(z_i z_j)^m}.
\]
Consequently,
\[
V_s(Z)^{2m}
=
(-1)^{m\binom s2}
e_s^{\,m(s-1)}D_{s,m}(Z). \tag{21}
\]
Since \(N=2m\) and \(L=m(s-1)+1\), equations (20)–(21) give
\[
[x^{a_{2m}}]V_N^N
=
(-1)^{m\binom s2}
\operatorname{CT}\left(D_{s,m}(Z)\frac{R(Z)}{e_s(Z)}\right). \tag{22}
\]

It remains to apply the lemma. From (15),
\[
\sum_{r\ge0}\epsilon_{-1/m}(E_r)t^r
=
\exp\left(
\sum_{j\ge1}
\frac{(-1)^{j-1}N(-1/m)t^j}{j}
\right).
\]
Because \(N=2m\), this equals
\[
\exp\left(-2\sum_{j\ge1}\frac{(-1)^{j-1}t^j}{j}\right)
=(1+t)^{-2}.
\]
Hence
\[
\epsilon_{-1/m}(E_r)=(-1)^r(r+1). \tag{23}
\]

Using (19), (23), and \(s=N-2\), we obtain
\[
\begin{aligned}
\epsilon_{-1/m}(R)
&=-(-1)^s\sum_{r=0}^s(-1)^r
  \binom N{r+1}(r+1)(s-r+1)\\
&=-(-1)^sN(N-1)
  \sum_{r=0}^{N-2}(-1)^r\binom{N-2}{r}\\
&=-(-1)^sN(N-1)(1-1)^{N-2}\\
&=0. \tag{24}
\end{aligned}
\]
Here we used
\[
\binom N{r+1}(r+1)(N-r-1)
=N(N-1)\binom{N-2}{r}.
\]
The constant-term lemma and (22) therefore give
\[
[x^{a_{2m}}]V_N^{2m}=0 \qquad(m\ge2). \tag{25}
\]

Thus a missing lattice point exists for every \(k\), with
\[
n(k)=
\begin{cases}
3,&k\ \text{odd},\\
6,&k=2,\\
k,&k\ge4\ \text{even}.
\end{cases} \tag{26}
\]

### Propagation

Suppose \(a\in kP_{\delta_n}\) and \([x^a]V_n^k=0\). Since
\[
V_{n+1}^k
=
V_n^k\prod_{i=1}^n(x_i-x_{n+1})^k,
\]
the vector
\[
\widetilde a=(a_1,\ldots,a_n,kn)
\]
lies in \(\operatorname{Newt}(V_{n+1}^k)\): in the Minkowski sum, add the exponent \((0,\ldots,0,kn)\) obtained by choosing \((-x_{n+1})^k\) from all \(n\) new factors.

Moreover, to obtain \(x_{n+1}\)-degree \(kn\), every new factor must contribute its unique term of \(x_{n+1}\)-degree \(k\). Therefore
\[
[x^{\widetilde a}]V_{n+1}^k
=
(-1)^{kn}[x^a]V_n^k
=0. \tag{27}
\]
Iteration proves that \(V_n^k\) is not SNP for every \(n\ge n(k)\). Hence Conjecture \(2.25\) holds.

> the codex process crashed and was restarted a couple of times so the stated final runtime of finding this proof is incorrect, estimate how much time codex took since the first prompt

<details><summary>1 previous message</summary>

> I’ll reconstruct this from the earliest surviving attachment/work-file timestamps and the goal metadata, then distinguish wall-clock elapsed time from active Codex compute time.
</details>

The proof took about **12 hours 41 minutes wall-clock** from the first prompt to the final answer:

- First prompt: July 22, 2026, 6:11:18 PM EDT
- Final proof delivered: July 23, 2026, 6:52:03 AM EDT

The recorded 2h20m only covered a surviving process segment. Excluding clearly identifiable crash/disconnection gaps, Codex appears to have been actively running for approximately **5½–6 hours**, with several subagents working concurrently.
\end{Verbatim}

\subsection{Initial prompt}
\begin{Verbatim}
You are directing a long-horizon mathematical research program whose sole objective is to prove Conjecture 2.25 of Monical–Tokcan–Yong, “Newton Polytopes in Algebraic Combinatorics,” arXiv:1703.02583.

NON-NEGOTIABLE OPERATING RULES

1. Assume that a complete affirmative proof exists.
2. Use as many as 64 agents dynamically.
3. Begin with independent, genuinely different approaches.
4. Keep a registry of approach families and their exact gaps.
5. Prevent early convergence on an attractive idea.
6. Require concrete lemmas, constructions, equations, or counterexamples—not progress reports.
7. Mark theorem-strength gaps as blocked.
8. Assign independent adversarial agents to attack every proposed proof.
9. Check known failure modes and circular reductions explicitly.
10. Continue launching new rounds after failures.
11. Spend at least eight hours before considering giving up.
12. Return only a complete proof that survives audit.

THE PROBLEM

For n ≥ 1, define

    δ_n = (n−1,n−2,...,1,0)

and the Vandermonde polynomial

    a_{δ_n}(x_1,...,x_n)
      = ∏_{1≤i<j≤n}(x_i−x_j).

A polynomial f has a saturated Newton polytope, abbreviated SNP, when

    supp(f) = Newton(f) ∩ Z^n.

Prove:

    For every positive integer k, there exists N_k such that
    a_{δ_n}^k is not SNP for every n ≥ N_k.

A complete proof need not determine the smallest possible N_k. It is sufficient to construct, for every k, some n(k) and some integer vector α(k) such that

    α(k) ∈ Newton(a_{δ_{n(k)}}^k)

but

    [x^{α(k)}] a_{δ_{n(k)}}^k = 0.

Then rigorously invoke the propagation lemma from the source paper:

    If a_{δ_n}^k is not SNP, then a_{δ_{n+1}}^k is not SNP.

LITERATURE CHECKPOINT

Before attempting the proof, independently update the literature search through the present date. Search exact statements, citation graphs, theses, author publication lists, Vandermonde-power expansions, Dyson-product coefficients, Jack-polynomial expansions, and subsequent SNP literature.

At minimum inspect:

    • Monical–Tokcan–Yong, arXiv:1703.02583.
    • Ballantine, arXiv:1201.4572.
    • Expanded Vandermonde powers, arXiv:1204.6003.
    • Lv–Xin–Zhou, arXiv:0706.1009.
    • Sills, arXiv:1812.05557.
    • Later papers citing Monical–Tokcan–Yong.

An alleged resolution counts only after the full argument has been reconstructed and independently audited. A paper resolving a different Monical–Tokcan–Yong conjecture is not evidence that Conjecture 2.25 has been resolved.

KNOWN FACTS TO VERIFY BEFORE USING

1. Newton polytope:

       Newton(a_{δ_n}^k) = kP_{δ_n},

   where P_{δ_n} is the permutahedron generated by the permutations of δ_n.

2. Majorization test:

   If α↓ is α rearranged into weakly decreasing order, then

       α ∈ kP_{δ_n} ∩ Z^n

   exactly when α is integral,

       ∑_{i=1}^n α_i = k n(n−1)/2,

   and, for every 1 ≤ r < n,

       ∑_{i=1}^r α_i↓
       ≤ k(rn − r(r+1)/2).

3. Propagation:

   Check every detail of Lemma 2.24, including why the coefficient with new exponent kn in x_{n+1} can only arise by choosing the x_{n+1}^k term from every new factor.

4. Computed cases reported in the source:

       N_2 = 5,
       N_4 = 4,
       N_6 = 4,
       N_8 = 3,

   and N_{2j−1}=3 was checked there for 1≤j≤4. Reproduce these cases by exact arithmetic; do not trust undocumented computations.

MANDATORY ODD-POWER REDUCTION

Immediately verify the following observation.

For odd k, let

       α_k = (1,(3k−1)/2,(3k−1)/2).

Prove:

    • α_k is an integer point of kP_{(2,1,0)}.
    • a_{δ_3}^k is alternating.
    • Swapping the two equal exponent coordinates fixes x^{α_k}
      but negates the polynomial.
    • Therefore [x^{α_k}]a_{δ_3}^k=0.

If correct, record the entire odd-k case as closed with N_k=3. Do not continue spending the majority of resources on odd k. Concentrate on k=2m.

EVEN-POWER NORMAL FORM

For k=2m define the Dyson Laurent polynomial

       D_{n,m}(x)
         = ∏_{i≠j}(1−x_i/x_j)^m.

Verify the exact identity

       a_{δ_n}^{2m}
         = (−1)^{m·binom(n,2)}
           (∏_{i=1}^n x_i^{m(n−1)}) D_{n,m}(x).

Consequently, for

       β = α − m(n−1)(1,...,1),

coefficient vanishing in a_{δ_n}^{2m} is equivalent, up to the displayed global sign, to coefficient vanishing in D_{n,m}. Exploit Dyson, Morris, Aomoto, Kadell, Good-recursion, and constant-term machinery where applicable, but do not confuse a formula for selected coefficients with a characterization of the entire support.

COEFFICIENT MODEL

Independently use the following bounded edge-label model.

For every edge i<j of K_n choose

       t_{ij} ∈ {0,1,...,k},

where t_{ij} records how many copies of −x_j are selected from
(x_i−x_j)^k. Then

       α_i
       = ∑_{j>i}(k−t_{ij}) + ∑_{h<i}t_{hi},

and

       [x^α]a_{δ_n}^k
       = ∑_t ∏_{i<j}
           (−1)^{t_{ij}} binom(k,t_{ij}),

where the sum is over all edge labels satisfying the exponent equations.

Search for exponent families for which this signed sum vanishes by:

    • a fixed-point-free sign-reversing involution;
    • an exact hypergeometric identity;
    • a determinant or Pfaffian identity;
    • a root-of-unity filter;
    • representation-theoretic orthogonality;
    • a recurrence with provably zero initial and boundary data;
    • an arithmetic decomposition covering every even k.

INITIAL APPROACH FAMILIES

Launch genuinely independent groups for at least these families:

A. Explicit universal holes
   Search directly for n=n(k) and α(k), possibly depending on the
   residue class or prime factorization of k.

B. Fixed small dimension
   Derive exact coefficient formulas for n=3, n=4, and n=5.
   Classify when these one- or low-dimensional binomial sums vanish.

C. Dyson constant terms
   Translate candidate holes into coefficients of D_{n,m} and test
   whether known generalized Dyson identities force zeros.

D. Signed complete-graph flows
   Interpret edge labels as bounded integer flows and construct
   weight-preserving sign-reversing involutions.

E. Schur and Jack expansions
   Use even-power symmetry, Jack-polynomial identities, and recursive
   Vandermonde expansions. Explicitly convert any conclusion back to
   a monomial coefficient.

F. Induction and embedding
   Investigate induction in m, n, prime powers, and multiplicative
   decompositions of k. Identify exactly what property is preserved.

G. Arithmetic and finite fields
   Use Lucas, Kummer, roots of unity, or Frobenius only when the result
   proves an integer coefficient is exactly zero in characteristic
   zero. Congruence modulo one prime is insufficient.

H. Exact computation and conjecture discovery
   Compute supports and coefficients using integers, not floating
   point. Search for symbolic patterns in missing exponent partitions.
   Produce machine-checkable certificates, then replace finite data
   with a uniform proof.

I. Adversarial falsification
   Try to disprove every candidate general lemma using exact searches
   at the smallest untested k and n.

APPROACH REGISTRY

For every approach family record internally:

    • precise target lemma;
    • parity or residue classes covered;
    • exact hypotheses;
    • strongest proved statement;
    • smallest unresolved instance;
    • smallest counterexample found;
    • exact theorem-strength gap;
    • whether the gap is local, computational, or structural;
    • adversarial audit status.

“Promising,” “likely,” “appears,” “numerically verified,” and
“standard methods should finish” are not acceptable registry states.

A family may be marked complete only after it supplies a lemma that
covers an explicit infinite class of k with every quantifier proved.
A family must be marked blocked when its remaining assertion is
equivalent to the original conjecture or to another unproved theorem
of comparable strength.

MANDATORY FAILURE-MODE AUDIT

Reject any proposed proof containing one of the following defects:

1. It proves only finitely many values of k.
2. It proves only that a coefficient vanishes modulo p.
3. It relies on floating-point cancellation.
4. It finds a zero coefficient without proving the exponent lies in
   kP_{δ_n}.
5. It proves polytope membership by assuming full monomial support.
6. It proves a Schur coefficient is zero and silently treats this as
   a monomial coefficient being zero.
7. It constructs an involution without proving it is defined
   everywhere, fixed-point-free, sign-reversing, and weight-preserving.
8. It invokes a constant-term identity outside its hypotheses.
9. It divides by a quantity that may vanish for some k.
10. It leaves an uncovered congruence class of even k.
11. It invokes Lemma 2.24 without checking the extremal
    x_{n+1}-coefficient argument.
12. It confuses finding some threshold N_k with proving that N_k is
    minimal.
13. It uses a recurrence without proving all boundary values and that
    the recurrence reaches every required parameter.
14. It assumes that a candidate exponent is in the Newton polytope
    merely because its coordinates have the correct total sum.
15. It hides the original conjecture inside phrases such as
    “generic cancellation,” “standard representation theory,” or
    “a routine coefficient calculation.”

ADVERSARIAL REVIEW

Every proposed final proof must be attacked by independent agents
assigned to:

    • verify every quantifier;
    • recompute signs and exponent shifts;
    • check the majorization inequalities;
    • search exact small counterexamples;
    • test all parity and residue-class boundaries;
    • verify all cited theorems from their original statements;
    • detect circular use of SNP;
    • distinguish monomial, Schur, and Jack bases;
    • confirm that characteristic-zero vanishing has actually been
      proved.

If an attack succeeds, return the proof to the registry with the exact
failed line and launch new approaches. Do not patch a failed proof by
assertion.

COMPLETION STANDARD

A surviving proof must explicitly provide, for every k≥1:

    1. A finite integer n(k).
    2. An explicit integer vector α(k), possibly defined casewise.
    3. A proof that α(k) ∈ kP_{δ_{n(k)}}.
    4. A characteristic-zero proof that
       [x^{α(k)}]a_{δ_{n(k)}}^k=0.
    5. A rigorous use of Lemma 2.24 showing non-SNP for every
       n≥n(k).
    6. A check that all positive integers k are covered.

The final response must be a self-contained mathematical proof.
Include definitions and every essential lemma. Computations may guide
discovery but may not replace uniform arguments. Do not return a
research diary, registry, partial theorem, literature report, list of
approaches, or statement that the conjecture remains open.

Return only a complete proof that survives audit.
\end{Verbatim}

\end{document}